\documentclass{amsart}
\usepackage{amssymb,amscd}
\DeclareMathSymbol{\twoheadrightarrow} {\mathrel}{AMSa}{"10}

\def\Q{{\mathbf Q}}
\def\Z{{\mathbf Z}}
\def\C{{\mathbf C}}

\def\F{{\mathbf F}}
                               \def\ST{{\mathbf S}}

\def\SS{{\mathbf S}}

\def\Gal{\mathrm{Gal}}

\def\Is{\mathrm{Isog}}

\def\End{\mathrm{End}}
\def\Aut{\mathrm{Aut}}
\def\Hom{\mathrm{Hom}}

\def\fchar{\mathrm{char}}

\def\GL{\mathrm{GL}}

        \def\Is{\mathrm{Isog}}
                                \def\IsP{\mathrm{Is}}

                       \def\Spec{\mathrm{Spec}}

\def\dim{\mathrm{dim}}

                             \def\TORS{\mathrm{TORS}}
                             \def\non{\mathrm{non}}
\def\Oc{{\mathcal O}}
\def\P{{\mathbf P}}
     \def\W{{\mathbf W}}
     \def\S{{\mathbf S}}

\def\q{{\mathfrak q}}
                          \def\p{{\mathfrak p}}

\def\Y{{\mathcal Y}}

\newtheorem{thm}{Theorem}[section]
\newtheorem{lem}[thm]{Lemma}
\newtheorem{cor}[thm]{Corollary}

\theoremstyle{definition}

\newtheorem{sect}[thm]{}

           \newtheorem{rem}[thm]{Remark}

\title[Abelian varieties over fields of finite characteristic]
{Abelian varieties over fields of finite characteristic}
\author[Yuri G. Zarhin]{Yuri G. Zarhin}
\thanks{This work was partially supported by a grant from the Simons Foundation (\#246625 to Yuri Zarkhin).}
\address{Department of Mathematics, Pennsylvania State University,
University Park, PA 16802, USA}

 \email{zarhin\char`\@math.psu.edu}
\begin{document}
\begin{abstract}
The aim of this paper is to extend our previous results about Galois
action on the torsion points of abelian varieties to the case of
(finitely generated) fields of characteristic $2$.
\end{abstract}

\maketitle
\section{Introduction}
Let $K$ be a field, $\bar{K}$ its algebraic closure, $\bar{K}_s
\subset \bar{K}$ the separable algebraic closure of $K$,
$\Gal(K)=\Gal(\bar{K}_s/K)=\Aut(\bar{K}/K)$ the absolute Galois
group of $K$.

Let $X$ be an abelian variety over $K$. Then we write $\End(X)$ for
its ring of $K$-endomorphisms  and $\End^0(X)$ for the
finite-dimensional semisimple $\Q$-algebra
 $\End(X)\otimes\Q$.
If $n$ is a positive integer that is not
divisible by $\fchar(K)$ then we write $X_n$ for the kernel of
multiplication by $n$ in $X(\bar{K})$; it is well known that $X_n$
is free $\Z/n\Z$-module of rank $2\dim(X)$ \cite{Mumford}, which is
a Galois submodule of  $X(\bar{K}_s)$. We write $\bar{\rho}_{n,X}$
for the corresponding (continuous) structure homomorphism
$$\bar{\rho}_{n,X}: \Gal(K)\to \Aut_{\Z/n\Z}(X_n) \cong \GL(2\dim(X),\Z/n\Z).$$
In particular, if $n=\ell$ is a prime different from $\fchar(K)$
then $X_{\ell}$ is a $2\dim(X)$-dimensional $\F_{\ell}$-vector space
provided with
$$\bar{\rho}_{\ell,X}:\Gal(K)\to \Aut_{\F_{\ell}}(X_{\ell})\cong
\GL(2\dim(X),\F_{\ell}).$$ We write
$$\tilde{G}_{\ell}=\tilde{G}_{\ell,X,K}$$ for the
image (subgroup)
$$\bar{\rho}_{\ell,X}(\Gal(K))\subset \Aut_{\F_{\ell}}(X_{\ell}).$$
By definition, $\tilde{G}_{\ell,X,K}$ is a finite subgroup of
$$\Aut_{\F_{\ell}}(X_{\ell})\cong \GL(2\dim(X),\F_{\ell}).$$
If $K(X_{\ell})$ is the field of definition of all points of order
$\ell$ on $X$ then it is a finite Galois extension of $K$ and the
corresponding Galois group $\Gal(K(X_{\ell})/K)$ is canonically
isomorphic to $\tilde{G}_{\ell,X,K}$. If $K^{\prime}/K$ is a finite
Galois extension of fields then $\Gal(K^{\prime})$ is a normal open
subgroup of finite index in $\Gal(K)$ while $X^{\prime}=X \times_K
K^{\prime}$ is a $\dim(X)$-dimensional abelian variety over
$K^{\prime}$ and the $\Gal(K^{\prime})$-modules $X_{\ell}$ and
$X^{\prime}_{\ell}$ are canonically isomorphic. Under this
isomorphism, $\tilde{G}_{\ell,X^{\prime},K^{\prime}}$ becomes
isomorphic to a certain normal subgroup of $\tilde{G}_{\ell,X,K}$.

 By functoriality, $\End(X)$  acts on $X_n$. This
action gives rise to the embedding of free $\Z/n\Z$-modules
$$\End(X)\otimes \Z/n\Z
\hookrightarrow \End_{\Z/n\Z}(X_n);$$ in addition, the image of
$\End(X)\otimes \Z/n\Z$ lies in the centralizer
$\End_{\Gal(K)}(X_n)$ of $\Gal(K)$ in $\End_{\Z/n\Z}(X_n)$. Further
we will identify $\End(X)\otimes \Z/n\Z$ with its image in
$\End_{\Gal(K)}(X_n)$ and write
$$\End(X)\otimes
\Z/n\Z\subset \End_{\Gal(K)}(X_n).$$ If $\ell$ is a prime that is
different from $\fchar(K)$ then we write $T_{\ell}(X)$ for the
$\Z_{\ell}$-Tate module of $X$ and $V_{\ell}(X)$ for the
corresponding $\Q_{\ell}$-vector space
$$V_{\ell}(X)=T_{\ell}(X)\otimes_{\Z_{\ell}}\Q_{\ell}$$
provided with the natural continuous Galois action \cite{Serre}
$$\rho_{\ell,X}:\Gal(K) \to \Aut_{\Z_{\ell}}(T_{\ell}(X))\subset
\Aut_{\Q_{\ell}}(V_{\ell}(X)).$$ Recall \cite{Mumford} that
$T_{\ell}(X)$ is a free $\Z_{\ell}$-module of rank $2\dim(X)$ and
$V_{\ell}(X)$ is a $\Q_{\ell}$-vector space of dimension $2\dim(X)$.
Notice that there are canonical isomorphisms of $\Gal(K)$-modules
$$X_{\ell^i}=T_{\ell}(X)/\ell^i T_{\ell}(X) \eqno(0)$$
for all positive integers $i$.  The natural embeddings
$$\End(X)\otimes \Z/\ell^i\Z
\hookrightarrow \End_{\Z/\ell^i\Z}(X_{\ell^i})$$ are compatible and
give rise to the embeddings of  $\Z_{\ell}$-algebras
$$\End(X)\otimes \Z_{\ell}
\hookrightarrow \End_{\Z_{\ell}}(T_{\ell}(X))$$ and
$\Q_{\ell}$-algebras
$$\End(X)\otimes \Q_{\ell}
\hookrightarrow \End_{\Q_{\ell}}(V_{\ell}(X)).$$ Again, the images
of $\End(X)\otimes \Z_{\ell}$ in $\End_{\Z_{\ell}}(T_{\ell}(X))$ and
of $\End(X)\otimes \Q_{\ell}$ in $\End_{\Q_{\ell}}(V_{\ell}(X))$ lie
in the centralizers $\End_{\Gal(K)}(T_{\ell}(X))$ and
$\End_{\Gal(K)}(V_{\ell}(X))$ respectively. We will identify
$\End(X)\otimes \Z_{\ell}$ and $\End(X)\otimes \Q_{\ell}$ with their
respective images and write
$$\End(X)\otimes \Z_{\ell}\subset \End_{\Gal(K)}(T_{\ell}(X))\subset \End_{\Z_{\ell}}(T_{\ell}(X)),$$
$$\End(X)\otimes \Q_{\ell}\subset \End_{\Gal(K)}(V_{\ell}(X))\subset
\End_{\Q_{\ell}}(V_{\ell}(X)).$$
Similarly, if $Y$ is  another abelian variety over $K$ then
we write $\Hom(X,Y)$ for the (free commutative) group of all
$K$-homomorphisms from $X$ to $Y$. Similarly,  there are the natural
embeddings
$$\Hom(X,Y) \otimes \Z/n\Z \subset \Hom_{\Gal(K)}(X_n,Y_n)\subset \Hom_{\Z/n\Z}(X_n,Y_n),$$
$$\Hom(X,Y) \otimes \Z/\ell^i\Z \subset
\Hom_{\Gal(K)}(X_{\ell^i},Y_{\ell^i})\subset
\Hom_{\Z/\ell^i\Z}(X_{\ell^i},Y_{\ell^i}),$$
$$\Hom(X,Y) \otimes \Z_{\ell}\subset
\Hom_{\Gal(K)}(T_{\ell}(X),T_{\ell}(Y))\subset
\Hom_{\Z_{\ell}}(T_{\ell}(X),T_{\ell}(Y)),$$
$$\Hom(X,Y) \otimes \Q_{\ell}\subset
\Hom_{\Gal(K)}(V_{\ell}(X),V_{\ell}(Y))\subset
\Hom_{\Q_{\ell}}(V_{\ell}(X),V_{\ell}(Y)).$$

 Let
$$\rho_{\ell,X}: \Gal(K) \to \Aut_{\Z_{\ell}}(T_{\ell}(X))\subset
\Aut_{\Q_{\ell}}(V_{\ell}(X))$$ be the corresponding $\ell$-adic
representation of $\Gal(K)$. The image
$$G_{\ell,X,K}=\rho_{\ell,X}(\Gal(K)) \subset \Aut_{\Z_{\ell}}(T_{\ell}(X))\subset
\Aut_{\Q_{\ell}}(V_{\ell}(X))$$ is a compact $\ell$-adic Lie
(sub)group  \cite{SerreIzv,Serre}.

Let $d$ be a positive integer. We write $\Is(X,K,d)$ for the set of
$K$-isomorphism classes of abelian varieties $Y$ over $K$ that enjoy
the following properties:

\begin{itemize}
\item[(i)]
$Y$ admits a $K$-polarization of degree $d$;
\item[(ii)]
There exists a $K$-isogeny $Y \to X$ whose degree is prime to
$\fchar(K)$.
\end{itemize}

For example, if $d=1$ then $\Is(X,K,1)$ consists of ($K$-isomorphism
classes of) all principally polarized abelian varieties $Y$ over $K$
that admit a $K$-isogeny whose degree is prime to $\fchar(K)$.

We write $\Is(X,K)$ for  the set of $K$-isomorphism classes of
abelian varieties $Y$ over $K$ such that there exists a $K$-isogeny
$Y \to X$ whose degree is prime to $\fchar(K)$. Clearly, $\Is(X,K)$
coincides with the union of all $\Is(X,K,d)$'s.

The following statement was proven under an additional assumption
that $p$ does not divide $d$ by the author when $p>2$
\cite{ZarhinMZ1} and by S. Mori when $p=2$ \cite[Ch. XII, Cor. 2.4
on p. 244]{MB}. (This is a strenghening of Tate' finiteness
conjecture for isogenies of abelian varieties \cite{Tate,ZarhinP}.)

\begin{thm}[Corollary 2.4 on p. 244 of \cite{MB}]
\label{finiteTate} Assume  that $p:=\fchar(K)>0$ and $K$ is finitely
generated over the finite prime field $\F_p$. Let $d$ be a positive
integer and $X$ be an abelian variety over $K$.

Then the set $\Is(X,K,d)$ is finite.
\end{thm}

The finiteness of $\Is(X,K,d)$ combined with results of
\cite{ZarhinIz} implies the Tate conjecture on homomorphisms of
abelian varieties and the semisimplicity of Tate modules over $K$
(see \cite{ZarhinMZ1} for $p>2$ and \cite[Ch. XII, Th. 2.5 on pp.
244--245]{MB}).

\begin{thm}[Theorem 2.5 on pp. 244--245 of \cite{MB}]
\label{TateP} Assume  that $p:=\fchar(K)>0$ and $K$ is finitely
generated over the finite prime field $\F_p$.

Then for all abelian varieties $A$ and $B$ over $K$   and every
prime $\ell \ne \fchar(K)$ the Galois module $V_{\ell}(A)$ is
semisimple and the natural embedding of $\Z_{\ell}$-modules
$$\Hom(A,B)\otimes \Z_{\ell} \hookrightarrow
\Hom_{\Gal(K)}(T_{\ell}(A),T_{\ell}(B))$$ is bijective.
\end{thm}

\begin{rem}
\label{dEqualsOne} In fact, Theorem \ref{TateP} follows even from a special case of Theorem \ref{finiteTate} that deals only with principally polarized abelian varieties (i.e., when $d=1$), see Remark \ref{finalTate}.
\end{rem}

By Lemma 1 of \cite[Sect. 1]{Tate} the second assertion of Theorem
\ref{TateP} implies the following statement.

\begin{thm}
\label{TatePV} Assume  that $p:=\fchar(K)>0$ and $K$ is finitely
generated over the finite prime field $\F_p$.

Then for all abelian varieties $A$ and $B$ over $K$    the natural
embedding of $\Q_{\ell}$-vector spaces
$$\Hom(A,B)\otimes \Q_{\ell} \hookrightarrow
\Hom_{\Gal(K)}(V_{\ell}(A),V_{\ell}(B))$$ is bijective.
\end{thm}

\begin{sect}
\label{jacobson} Let $K$ be a field that is finitely generated over
the finite prime field $\F_p$ and $A$ an abelian variety of positive
dimension over $K$. Let $\ell$ be a prime different from $p$. By
Theorem \ref{TatePV} (applied to $B=A$) and Theorem \ref{TateP}, the
$\Gal(K)$-module $V_{\ell}(A)$ is semisimple and
$$\End_{\Gal(K)}(V_{\ell}(A))=\End(A)\otimes\Q_{\ell}=\End^0(A)\otimes_{\Q}\Q_{\ell}.$$
Since $G_{\ell,A,K}$ is the image of $\Gal(K) \to
\Aut_{\Z_{\ell}}(T_{\ell}(A))\subset \Aut_{\Q_{\ell}}(V_{\ell}(A))$,
the $G_{\ell,A,K}$-module $V_{\ell}(A)$ is semisimple and
$$\End_{G_{\ell,A,K}}(V_{\ell}(A))=\End(A)\otimes\Q_{\ell}=\End^0(A)\otimes_{\Q}\Q_{\ell}.$$
Let $\Q_{\ell}G_{\ell,A,K}$ be the $\Q_{\ell}$-subalgebra of
$\End_{\Q_{\ell}}(V_{\ell}(A))$ spanned by $G_{\ell,A,K}$. It
follows from the Jacobson density theorem \cite[Ch. XVII, Sect. 3, Th. 1]{Lang} that
$\Q_{\ell}G_{\ell,A,K}$ coincides with the centralizer of
$\End(A)\otimes\Q_{\ell}$ in $\End_{\Q_{\ell}}(V_{\ell}(A))$. It
follows easily that if $\Z_{\ell}G_{\ell,A,K}$ is the
$\Z_{\ell}$-subalgebra of $\End_{\Z_{\ell}}(T_{\ell}(A))$ spanned by
$G_{\ell,A,K}$ then the centralizer of $\End(A)\otimes\Z_{\ell}$ in
$\End_{\Z_{\ell}}(T_{\ell}(A))$ contains $\Z_{\ell}G_{\ell,A,K}$ as
a $\Z_{\ell}$-submodule of finite index.
\end{sect}

\section{Main results}
The aim of this note is to prove  variants of Theorem
\ref{TateP} where Tate modules are replaced by Galois modules
$A_{n}$ and $B_{n}$. Most of our results were already proven in
\cite{ZarhinMZ2, ZarhinJussieu} or stated in \cite{ZarhinYar} under
an additional assumption $p>2$. (See also \cite{ZarhinG} where the
case of finite fields is discussed.)

Throughout this section, $K$ is a field that is finitely  generated
over the finite prime field $\F_p$.

\begin{thm}
\label{finiteisogp} Let $A$ be an
  abelian variety of positive dimension over $K$. Then
the set $\Is(A,K)$ is finite.
\end{thm}

\begin{rem}
A  weaker version of Theorem \ref{finiteisogp} (where $\Is(A,K)$
 is replaced by its  subset that consists of abelian varieties that admit a polarization of degree prime to $p$)
  is proven in \cite[Th.
6.1]{ZarhinMZ2} under an additional assumption that $p>2$.
\end{rem}

\begin{cor}
\label{semisimpleF} Let $A$ be an abelian variety of positive
dimension over  $K$. There exists a positive integer $r=r(A)$ that
is not divisible by $p$ and enjoys the following properties.

\begin{itemize}
\item[(i)]
If $C$ is an abelian variety over $K$ that admits a $K$-isogeny $C
\to A$ whose degree is not divisible by $p$ then there exists a
$K$-isogeny $\beta: A \to C$ with $\ker(\beta) \subset A_r$.
\item[(ii)]
If $n$ is a positive integer that is not divisible by $p$ and $W$ is
a Galois submodule in $A_n$ then there exists $u \in \End(A)$ such
that
$$r W \subset u(A_n)\subset W.$$
\item[(iii)]
For all but finitely many primes $\ell$ the Galois module $A_{\ell}$
is semisimple.
\end{itemize}
\end{cor}

\begin{rem}
Corollary \ref{semisimpleF}(iii) is proven in \cite[Th.
1.1]{ZarhinMZ2} under an additional assumption $p>2$.
\end{rem}

\begin{thm}
\label{endoN} Let $A$ be an abelian variety  over $K$. Then there
exists a positive integer $r_1=r_1(A)$ that enjoys the following
properties.

Let $n$ be a positive integer that is not divisible by $p$ and $u_n$
an endomorphism of the Galois module $A_n$. If we put $m:=n/(n,r_1)$
then there exists $u \in \End(A)$ such that both $u$ and $u_n$
 induce the same endomorphism of the Galois module $A_m$.
\end{thm}

If $A$ and $B$ are abelian varieties over $K$ then applying Theorem
\ref{endoN} to their product $X=A\times B$, we obtain the following
statement.

\begin{thm}
\label{homoN} Let $A$ and $B$ be  abelian varieties  over $K$. Then
there exists a positive integer $r_2=r_2(A,B)$ that enjoys the
following properties.

Let $n$ be a positive integer that is not divisible by $p$ and $u_n:
A_n \to B_n$ a homomorphism of the Galois modules. If we put
$m:=n/(n,r_2)$ then there exists $u \in \Hom(A,B)$ such that both
$u$ and $u_n$  induce the same homomorphism of the Galois modules
$A_m\to B_m$.
\end{thm}

Theorem \ref{homoN} implies readily the following assertion.

\begin{cor}
\label{tateFiniteC} Let $A$ and $B$ be abelian varieties over $K$.
Then for all but finitely many primes $\ell$ the natural injection
$$\Hom(A,B)\otimes \Z/\ell\Z \hookrightarrow
\Hom_{\Gal(K)}(A_{\ell},B_{\ell})$$ is bijective.
\end{cor}

\begin{rem}
Theorem \ref{homoN} was stated without proof in \cite{ZarhinYar}
under an additional condition that $p>2$. Corollary
\ref{tateFiniteC} was proven in \cite[Th. 1.1]{ZarhinMZ2} under an
additional condition that $p>2$.
\end{rem}

\begin{thm}
\label{jacobsonT} Let $A$ be an abelian variety of positive
dimension over $K$. Then for all but finitely many primes $\ell$ the
centralizer of $\End(A)\otimes\Z_{\ell}$ in
$\End_{\Z_{\ell}}(T_{\ell}(A))$ coincides with
$\Z_{\ell}G_{\ell,A,K}$.
\end{thm}

\begin{rem}
When $K$ is a field of characteristic zero that is finitely
generated over $\Q$, an analogue of Theorem \ref{jacobsonT} was
proven by G. Faltings \cite[Sect. 3, Th. 1(c)]{Faltings2}.
\end{rem}

Recall that an old result of Grothendieck \cite{OortSM} asserts that
in characteristic $p$ an abelian variety of CM-type is isogenous to
an abelian variety that is defined over a finite field. (The
converse follows from a theorem of Tate \cite{Tate}.)

\begin{thm}
\label{CM} Let $X$ be an abelian variety over $K$. Suppose that for
infinitely many primes $\ell$ the group $\tilde{G}_{\ell,X,K}$ is
commutative. Then $X$ is an abelian variety of CM type over $K$ and
therefore is isogenous over $\bar{K}$ to an abelian variety that is
defined over a finite field.
\end{thm}

Theorem \ref{CM} may be strengthened as follows.

\begin{thm}
\label{fongswan} Let $X$ be an abelian variety over $K$. Suppose
that for infinitely many primes $\ell$ the group
$\tilde{G}_{\ell,X,K}$ is $\ell$-solvable, i.e., its
Jordan--H\"older factors are either $\ell$-groups or groups whose
order is not divisible by $\ell$. Then there is a finite Galois
extension $K^{\prime}/K$ such that $X \times_{K} K^{\prime}$ is an
abelian variety of CM type over $ K^{\prime}$ and therefore is
isogenous over $\bar{K}$ to an abelian variety that is defined over
a finite field.
\end{thm}

Theorem \ref{fongswan} combined with the celebrated theorem of
Feit-Thompson (about solvability of groups of odd order) implies
readily the following statement.

\begin{cor}
Let $X$ be an abelian variety of positive dimension over $K$ that is
not isogenous over $\bar{K}$ to an abelian variety that is defined
over a finite field. Then for all but finitely many primes $\ell$
the group $\tilde{G}_{\ell,X,K}$ is not solvable and its order is
divisible by $2\ell$.
\end{cor}

\begin{rem}
See \cite{ZarhinMRL2, ZarhinBSMF, ZarhinSb} for  plenty of explicit
examples of abelian varieties in characteristic $p$ without CM.

 Theorem \ref{CM} was proven in \cite{ZarhinMZ3} under an
additional assumption that $p>2$. Theorem \ref{fongswan}  was stated
without proof in \cite{ZarhinYar} under an additional assumption
that $K$ is global and $p>2$.
\end{rem}

In order to state  another (partial) strenghening of Theorem \ref{CM}, we need to introduce the following notation.
If $\mathcal{A}$ a commutative group then we write $\TORS(\mathcal{A})$ for its subgroup of all periodic elements and
$\TORS(\mathcal{A})(\non-p)$ that consists of all elements of $\TORS(\mathcal{A})$, whose order is prime to $p$.

\begin{thm}
\label{maxab} Let $K^{ab}\subset \bar{K}_s$ be the maximal abelian
extension of $K$. Let $X$ be a simple abelian variety over $K$. If
$\TORS(X(K^{ab}))(\non-p)$ is infinite then $X$ is an abelian
variety of CM-type over $K$ and therefore is isogenous over
$\bar{K}$ to an abelian variety that is defined over a finite field.
\end{thm}

\begin{rem}
In characteristic zero an analogue of Theorem \ref{maxab} was proven
in \cite[Th. 1.5]{ZarhinDUKE}.
\end{rem}

Theorem \ref{maxab} implies readily the following statement.
(Compare with \cite[Corollary on p. 132]{ZarhinDUKE}.)

\begin{cor}
\label{maxabNS} Let $K^{ab}\subset \bar{K}_s$ be the maximal abelian
extension of $K$. Let $X$ be an abelian variety of positive
dimension over $K$. Let $X_1, \dots , X_r$ be simple abelian
varieties over $K$ such that the product $\prod_{i=1}^r X_i$ is
$K$-isogenous to $X$. Then $\TORS(X(K^{ab}))(\non-p)$ is finite if
and only if all the groups $\TORS(X_i(K^{ab}))(\non-p)$ are finite,
i.e., all $X_i$ are not of CM-type over $K$ ($1\le i\le r)$.
\end{cor}

Now we discuss the torsion of abelian varieties in infinite Galois
extensions of $K$ with finite ``field of constants".

\begin{thm}
\label{finiteConst} Let $X$ be an abelian variety of positive
dimension over $K$ such that the center of $\End^0(X)$ is a direct
sum of totally real number fields. Let $K^{\prime}\subset \bar{K}_s$
be an infinite Galois extension of $K$. Let $\F^{\prime}$ be the
algebraic closure of $\F_p$ in $K^{\prime}$ and suppose that
$\F^{\prime}$ is a finite field. Then $\TORS(X(K^{\prime}))(\non-p)$
is finite.
\end{thm}

Theorem \ref{finiteConst} is an immediate corollary of the
conjunction of following two assertions.

\begin{thm}
\label{ConstEll} Let $X$ be an abelian variety of positive dimension
over $K$ such that the center of $\End^0(X)$ is a direct sum of
totally real number fields. Let $K^{\prime}\subset \bar{K}_s$ be an
infinite Galois extension of $K$. If $\ell \ne p$ is a prime such
that the $\ell$-primary component of $\TORS(X(K^{\prime}))$ is
infinite then $K^{\prime}$ contains all $\ell$-power roots of unity.
In particular, the algebraic closure of $\F_p$ in $K^{\prime}$ is
infinite.
\end{thm}

\begin{thm}
\label{ConstRoot} Let $X$ be an abelian variety of positive
dimension over $K$ such that the center of $\End^0(X)$ is a direct
sum of totally real number fields. Let us choose a polarization
$\lambda: X \to X^t$ that is defined over $K$. Let $\ell \ne p$ be a
prime that enjoys the following properties:

\begin{itemize}
\item[(i)]
$\ell$ is odd and prime to $\deg(\lambda)$;
\item[(ii)]
The $\Gal(K)$-module $X_{\ell}$ is semisimple and
$$\End_{\Gal(K)}(X_{\ell})=\End(X)\otimes \Z/\ell\Z.$$
\item[(iii)] If $C$ is the center of $\End(X)$ then $C/\ell C$ is the
center of $\End(X)/\ell \End(X)$.
\end{itemize}
If the $\ell$-primary component of $\TORS(X(K^{\prime}))$ does not
vanish then $K^{\prime}$ contains a primitive $\ell$th root of
unity. In particular, if $\F^{\prime}$ is the algebraic closure of
$\F_p$ in $K^{\prime}$ then its order  is strictly greater than
$\ell$.
\end{thm}

\begin{rem}
Let $S$ be the set of primes $\ell$ that do not enjoy either
property (i) or property (ii). Then $S$ is finite.
\end{rem}

The paper is organized as follows. In Section \ref{help} we discuss
isogenies of abelian varieties and their kernels (viewed as finite
Galois modules). One of the goals of our approach is to stress the
role of analogues of Tate's finiteness conjecture for isogeny
classes of abelian varieties.  In Section \ref{mainProof} we prove
all the main results except Theorem
\ref{fongswan}, which will be proven in Section \ref{lifting}.
Section \ref{conclude} contains additional references to results
that may be extended to characteristic $2$ case.

{\bf Acknowledgements}. I am grateful to Alexey Parshin, Chad Schoen
and Doug Ulmer for their interest in this paper. My special thanks
go to the referees, whose comments helped to improve the exposition.

The final version of this paper was prepared during my stay at Max-Planck-Institut f\"ur Mathematik (Bonn) in September 2013: I am grateful to the MPI for the hospitality and support.

\section{Isogenies and finite Galois modules}
\label{help} We write $\P$ for the set of all primes. Let $K$ be a
field. Let $P\subset \P$ be a nonempty set of primes that does {\sl
not} contain $\fchar(K)$. If $X$ and $Y$ are abelian varieties over
$K$ then a $K$-isogeny $X \to Y$ is called a $P$-isogeny if all
prime divisors of its degree are elements of $P$. For example, if
$P$ is a singleton $\{\ell\}$ then a $P$-isogeny is nothing else but
an $\ell$-power isogeny. We say that $X$ and $Y$ are $P$-isogenous
over $K$ if there is a $P$-isogeny $X \to Y$ that is defined over
$K$. The property to be $P$-isogenous is an equivalence relation.
Indeed, one has only to check that there exists a $P$-isogeny $v:Y
\to X$ that is defined over $K$. Indeed, thanks to Lagrange theorem,
$\ker(u) \subset X_n$ where $$n: =\deg(u)=\#(\ker(u)).$$ It follows
that there is a $K$-isogeny $v: Y \to X$ such that the composition
$vu:X \to Y \to X$ coincides with multiplication by $n$ in $X$. This
implies that
$$n^{2\dim(X)}=\deg(vu)=\deg(v)\deg(u).$$
Since $u$ is a $P$-isogeny, all prime divisors of $n$ belong to $P$.
This implies that all prime divisors of $\deg(v)$ also belong to
$P$, i.e., $v$ is a $P$-isogeny and we are done.

Let $X^t$ and $Y^t$ be the dual abelian varieties (over $K$) of $X$
and $Y$ respectively and
$$u^t: Y^t \to X^t, \ v^t: X^t \to Y^t$$
be the $K$-isogenies that are duals of $u$ and $v$ respectively.
Since
$$\deg(u^t)=\deg(u). \deg(v^t)=\deg(v),$$
$X^t$ and $Y^t$ are also $P$-isogenous over $K$. (Warning: $X$ and
$X^t$ are {\sl not} necessarily $P$-isogenous!) This implies that if
$X$ and $Y$ are $P$-isogenous over $K$ then $(X \times X^t)^4$ and
$(Y \times Y^t)^4$ are also $P$-isogenous over $K$.

We write $\IsP_P(X,K)$ for the set of isomorphism classes of abelian
varieties $Y$ over $K$ that are $P$-isogenous to $X$ over $K$. We
write $\IsP_P(X,K,1)$ for the subset of $\Is_P(X,K)$ that consists
of all isomorphism classes of $Y$ with principal polarization over
$K$. For example, if $P$ is $\P \setminus \{\fchar(K)\}$ then
$$\IsP_P(X,K)=\Is(X,K), \ \IsP_P(X,K,1)=\Is(X,K,1).$$

Now Theorem \ref{finiteisogp} becomes an immediate corollary of
Theorem \ref{finiteTate} and the following statement.

\begin{thm}
\label{finiteisog} Let $X$ be an  abelian variety   over a field
$K$. Suppose that  the set $\IsP_P((X\times X^t)^4,K,1)$ is finite.
Then the set $\IsP_P(X,K)$ is also finite.
\end{thm}

\begin{proof}
(i) Let us fix  an abelian variety $X$ be over $K$. Let $Y$ be an
abelian variety over $K$ that is $P$-isogenous to $X$ over $K$. As
we have seen, $(Y \times Y^t)^4$ is $P$-isogenous to $(X \times
X^t)^4$ over $K$. Recall \cite{ZarhinMZ2,ZarhinInv,MB}
 (see also
\cite[Sect. 7]{ZarhinG}) that $(Y \times Y^t)^4$ admits a principal
polarization over $K$ \footnote{In \cite[Ch. IX, Sect. 1]{MB}
Deligne's proof is given.}. Since the set $\Is_P((Y \times
Y^t)^4,K,1)$ is finite, the set of $K$-isomorphism classes of all
$(Y \times Y^t)^4$ (with fixed $X$) is finite. On the other hand,
each $Y$ is isomorphic to a $K$-abelian subvariety of $(Y \times
Y^t)^4$ over $K$. But the set of isomorphism classes of abelian
subvarieties of a given abelian variety is finite \cite{LOZ}. This
implies that the set of $K$-isomorphism classes of all $Y$'s is
finite.
\end{proof}

\begin{cor}
\label{almostss}
 Let $X$ be an abelian variety of positive dimension over a
field $K$. Suppose that  the set $\IsP_P((X\times X^t)^4,K,1)$ is
finite.

Then there exists a positive integer $r=r(X)$ that is not divisible
by $\fchar(K)$ and enjoys the following properties.

\begin{itemize}
\item[(i)]
If $Y$ is an abelian variety over $K$ that is $P$-isogenous to $X$
over $K$ then there exist a $P$-isogeny $\beta: X \to Y$ over $K$
with $\ker(\beta) \subset X_r$.
\item[(ii)]
If $n$ is a positive integer, all whose prime divisors lie in $P$
 and $W$ is a Galois submodule
in $X_n$, then there exists $u \in \End(X)$ such that
$$r W \subset u(X_n)\subset W.$$
\end{itemize}
\end{cor}

\begin{proof}
By Theorem \ref{finiteisog}, there are finitely many $K$-abelian
varieties $Y_1, \dots, Y_d$ that are $P$-isogenous to $X$ over $K$
and such that every $K$-abelian variety $Y$  that is
$P$-isogenous to $X$ over $K$ is $K$-isomorphic to one of $Y_j$. For
each $Y_i$ pick a $P$-isogeny $\beta_i: X \to Y_i$ that is defined
over $K$. Clearly, $\ker(\beta_i)\subset X_{m_i}$ where
$m_i=\deg(\beta_i)$. Let us put $r=\prod_{i=1}^{d} m_i$. Clearly,
for all $Y_i$
$$\ker(\beta_i)\subset X_{m_i}\subset X_r.$$
This implies that for every $K$-abelian variety $Y$  that is
$P$-isogenous to $X$ over $K$ there exists a $P$-isogeny $\beta:
X\to Y$ over $K$ whose kernel lies in $X_{r}$. This proves (i),
since every prime divisor of $r$ is a prime divisor of one of
$m_i=\deg(\beta_i)$ and therefore lies in $P$.

Proof of (ii) The quotient $Y=X/W$ is an abelian variety over $K$.
The canonical map $\pi: X\to X/W=Y$ is a $P$-isogeny over $K$,
because $\deg(\pi)=\#(W)$ divides $\#(X_n)=n^{2\dim(X)}$. This
implies that $Y$ is $P$-isogenous to $X$ over $K$.

The rest of the proof goes literally (with the same notation) as in
\cite[Sect. 8, pp. 331--332]{ZarhinG} provided one replaces the
reference to \cite[Cor. 3.5(i)]{ZarhinG} by the already proven case
(i) of Corollary \ref{almostss}. (In \cite{ZarhinG}, $n_X: X \to X$
and $n_Y:Y \to Y$ denote the multiplication by $n$ in $X$ and $Y$
respectively.)
\end{proof}

\begin{thm}
\label{ssL} Suppose that $P$ is infinite. Let $X$ be an abelian
variety of positive dimension over a field $K$. Suppose that  the
set $\IsP_P((X\times X^t)^4,K,1)$ is finite.

Then for all but finitely many primes $\ell\in P$ the Galois module
$X_{\ell}$ enjoys the following properties.

If $W$ is a Galois submodule in $X_{\ell}$ then there exists
$\tilde{u} \in \End(X)\otimes \Z/\ell Z$ such that
$\tilde{u}^2=\tilde{u}$ and $\tilde{u}(X_{\ell})=W$. In particular,
the Galois module $X_n$ splits into a direct sum
$$X_{\ell}=\tilde{u}(X_{\ell})\oplus (1-\tilde{u})(X_{\ell})=W\oplus
(1-\tilde{u})(X_{\ell})$$ of its Galois submodules $W$ and
$(1-\tilde{u})(X_{\ell})$.
\end{thm}

Theorem \ref{ssL} implies immediately the following assertion.

\begin{cor}
\label{ssLcor} Suppose that $P$ is infinite. Let $X$ be an abelian
variety of positive dimension over a field $K$. Suppose that  the
set $\IsP_P((X\times X^t)^4,K,1)$ is finite.

Then for all but finitely many primes $\ell\in P$ the Galois module
$X_{\ell}$ is semisimple.
\end{cor}

\begin{proof}[Proof of Theorem \ref{ssL}] It is well known that for
all but finitely many primes $\ell$ the finite-dimensional
$\F_{\ell}$-algebra $\End(X)\otimes\Z/\ell\Z$ is semisimple. (See,
e.g., \cite[Lemma 3.2]{ZarhinMZ2}.) Let $r$ be as in Corollary
\ref{almostss}. Now let $\ell \in P$ be a prime that does not divide
$r$ and such that $\End(X)\otimes\Z/\ell\Z$ is semisimple. Let $W$
be a Galois submodule in $X_{\ell}$. By Corollary \ref{almostss},
there exists $u \in \End(X)$ such that
$$rW \subset u(X_{\ell})\subset W.$$
Since $\ell$ does not divide $r$, we have $rW=W$ and therefore
$u(X_{\ell})=W$. Let $u_{\ell}$ be the image of $u$ in
$\End(X)\otimes\Z/\ell\Z$. Clearly,
$$u_{\ell}(X_{\ell})=u(X_{\ell})=W.$$
Let $I$ be the right ideal in semisimple $\End(X)\otimes\Z/\ell\Z$
generated by $u_{\ell}$. The semisimplicity implies that there
exists an {\sl idempotent} $\tilde{u}$ that generates $I$. It
follows that
$$W=u_{\ell}(X_{\ell})=\tilde{u}(X_{\ell}).$$
\end{proof}

 We will need the following lemma
\cite[Lemma 9.2 on p. 333]{ZarhinG}.

\begin{lem}
\label{lemma92} Let $Y$ be an abelian variety of positive dimension
over an arbitrary field $K$. Then there exists a positive integer
$h=h(Y,K)$ that enjoys the following properties.

If $n$ is a positive integer that is not divisible by $\fchar(K)$,
$u,v \in \End(Y)$ are endomorphisms of $Y$ such that
$$\{\ker(u)\bigcap Y_n\}\subset \{\ker(v)\bigcap Y_n\}$$
then there exists a $K$-isogeny $w: Y \to Y$ such that
$$hv-wu \in n \cdot \End(Y).$$
In particular, the images of $hv$ and $wu$ in
$$\End(Y)\otimes \Z/n\Z\subset\End_{\Gal(K)}(Y_n)\subset \End_{\Z/n\Z}(Y_n)$$
coincide.
\end{lem}

\begin{thm}
\label{endoNP} Let $X$ be an abelian variety of positive dimension
over a field $K$. Suppose that  the set $\IsP_P((X\times
X^t)^4,K,1)$ is finite. Then there exists a positive integer
$r_1=r_1(X,K)$ that enjoys the following properties.

Let $n$ be a positive integer, all whose prime divisors lie in $P$
and $m=n/(n,r_1)$. If $u_n \in \End_{\Gal(K)}(X_n)$ then there
exists $u \in \End(X)$ such that the images of $u_n$ and $u$ in
$\End_{\Gal(K)}(X_m)$ coincide.
\end{thm}

\begin{proof}
Let us put $Y=X\times X$. Then $(Y\times Y^t)^4=(X \times X^t)^8$.
Let $r(Y)$ be as in Corollary \ref{almostss} and $h(Y)$  as in Lemma
\ref{lemma92}. Let us put
$$r_1=r_1(X,K)=r(Y,K) h(Y,K).$$
Now the proof goes literally as the the proof of Theorem 4.1 in
\cite[Sect. 10]{ZarhinG}, provided one replaces the references to
Cor. 3.5 and Lemma 9.2 of \cite{ZarhinG} by references to Cor.
\ref{almostss} and Lemma \ref{lemma92} respectively.
\end{proof}

Let $A$ and $B$ be abelian varieties over $K$. Applying Theorem
\label{endoNF} to $X=A\times B$ and using the obvious compatible
decompositions
$$\End(X)=\End(A)\oplus \End(B)\oplus \Hom(A,B)\oplus \Hom(B,A),$$
$$\End_{\Gal(K)}(X_n)=$$
$$\End_{\Gal(K)}(A_n)\oplus \End_{\Gal(K)}(B_n)
\oplus \Hom_{\Gal(K)}(A_n,B_n)\oplus \Hom_{\Gal(K)}(B_n,A_n),$$ we
obtain the following statement.

\begin{thm}
\label{homoNP} Let $A$  and $B$ be  abelian varieties of positive
dimension over a field $K$. Suppose that  the set $\Is_P((A\times
B\times A^t\times B^t)^8,K,1)$ is finite. Then there exists a
positive integer $r_2=r_2(A,B,K)=r_1(A\times B,K)$ that enjoys the
following properties.

Let $n$ be a positive integer, all whose prime divisors lie in $P$
and $m=n/(n,r_1)$. If $u_n \in \Hom_{\Gal(K)}(A_n,B_n)$ then there
exists $u \in \Hom(A,B)$ such that the images of $u_n$ and $u$ in
$\Hom_{\Gal(K)}(A_m,B_m)$ coincide.
\end{thm}

\begin{cor}
\label{TateHomo} Let $A$  and $B$ be  abelian varieties of positive
dimension over a field $K$. Suppose that  the set $\Is_P((A\times
B\times A^t\times B^t)^8,K,1)$ is finite. Then for all primes $\ell
\in P$ the natural injection
$$\Hom(A,B)\otimes \Z_{\ell} \hookrightarrow
\Hom_{\Gal(K)}(T_{\ell}(A),T_{\ell}(B))$$ is bijective.
\end{cor}
\begin{proof}[Proof of Corollary \ref{TateHomo}]
Let $r_2=r_2(A,B)$ be as in Theorem \ref{homoNP}. Let $\ell^{i_0}$
be the exact power of $\ell$ that divides $r_2$.  Let $v \in
\Hom_{\Gal(K)}(T_{\ell}(A),T_{\ell}(B))$. For each $i>i_0$, $v$
induces a homomorphism $v_i \in \Hom_{\Gal(K)}(A_{\ell^i},
B_{\ell^i})$. By Theorem \ref{homoNP}, there exists $u_i \in
\Hom(A,B)$ such that the images of $u_i$ and $v_i$ in
$\Hom(A_{\ell^{i-i_0}}, B_{\ell^{i-i_0}})$ coincide. This means that $u_i-v$
sends $T_{\ell}(A)$ into $\ell^{i-i_0}T_{\ell}(B)$. It follows that
$v$ coincides with the limit of the sequence
$\{u_i\}_{i>i_0}^{\infty}$ in
$\Hom_{\Z_{\ell}}(T_{\ell}(A),T_{\ell}(B))$ with respect to
$\ell$-adic topology. Since $\Hom(A,B)\otimes \Z_{\ell}$ is a
compact and therefore a closed subset of
$\Hom_{\Z_{\ell}}(T_{\ell}(A),T_{\ell}(B))$, the limit $v$ also lies
in $\Hom(A,B)\otimes \Z_{\ell}$.
\end{proof}

The following lemma will be proven at the end of this section.

\begin{lem}
\label{endoTP} Let $X$ be an abelian variety of positive dimension
over a field $K$. Suppose that  the set $\IsP_P((X\times
X^t)^4,K,1)$ is finite. Let $r=r(X,K)$ be as in Corollary
\ref{almostss}. Then every $\ell \in P$ enjoys the following
properties.

 Let  $\ST$ be a Galois-invariant
$\Z_{\ell}$-submodule in $T_{\ell}(X)$ such that the quotient $T_{\ell}(X)/\ST$ is torsion-free. Then there exists $u\in
\End(X)\otimes \Z_{\ell}$ such that
$$ r_1 \cdot \ST \subset u(T_{\ell}(X)) \subset \ST.$$
\end{lem}

\begin{thm}
\label{ssTateP} Let $X$ be an abelian variety of positive dimension
over a field $K$. Suppose that  the set $\IsP_P((X\times
X^t)^8,K,1)$ is finite. Then every $\ell \in P$ enjoys the following
properties.

If $\W$ is a $\Gal(K)$-invariant $\Q_{\ell}$-vector subspace in
$V_{\ell}(X)$  then there exists $\tilde{u} \in \End(X)\otimes
\Q_{\ell}$ such that $\tilde{u}^2=\tilde{u}$ and
$\tilde{u}(V_{\ell}(X))=\W$. In particular,
$V_{\ell}(X)$ splits into a direct sum
$$V_{\ell}(X)=\tilde{u}(V_{\ell}(X))\oplus (1-\tilde{u})(V_{\ell}(X))=\W\oplus
(1-\tilde{u})(V_{\ell}(X))$$ of its Galois submodules $\W$ and
$(1-\tilde{u})(V_{\ell}(X))$ and the $\Gal(K)$-module $V_{\ell}(X)$
is semisimple.
\end{thm}

\begin{proof}[Proof of Theorem \ref{ssTateP}]
Let us put $\ST=\W\bigcap T_{\ell}(X)$.  Clearly, $\ST$ is a  Galois-invariant free $\Z_{\ell}$-submodule in $T_{\ell}(X)$
and $\W=\Q_{\ell}\ST$. In addition,  the quotient $T_{\ell}(X)/\ST$ is torsion-free.

By Lemma \ref{endoTP}, there exists $u\in
\End(X)\otimes \Z_{\ell}$ such that
$$ r \cdot \ST \subset u(T_{\ell}(X)) \subset \ST.$$
It follows that
$$r \cdot \W \subset u(V_{\ell}(X)) \subset \W.$$
Since $r \cdot \W=\W$, we have $u(V_{\ell}(X))=\W$. Notice that
$$\End(X)\otimes \Z_{\ell}\subset \End(X)\otimes
\Q_{\ell}=\End^0(X)\otimes_{\Q}\Q_{\ell}$$ and $\End^0(X)$ is a
finite-dimensional semisimple $\Q$-algebra. It follows that
$\End(X)\otimes \Q_{\ell}$ is a finite-dimensional semisimple
$\Q_{\ell}$-algebra. Let $I$ the left ideal in semisimple
$\End(X)\otimes \Q_{\ell}$ generated by $u$; there is an idempotent
$\tilde{u}$ that generates $I$. Clearly.
$$\tilde{u}(V_{\ell}(X))=u(V_{\ell}(X))=\W.$$
\end{proof}

\begin{proof}[Proof of Lemma \ref{endoTP}]
If $\SS=\{0\}$ then we just put $u=0$. If $\S=T_{\ell}(X)$ then we
take as $u$ the identity automorphism $1_X$ of $X$.

So, further we assume that $\SS$ is a proper free $\Z_{\ell}$-module
in $T_{\ell}(X)$  of positive rank say, $d$ and let $\{e_1, \dots
,e_j \dots e_d\}$ be its basis. Since $\SS$ is pure in
$T_{\ell}(X)$, for all positive integers $i$ the natural
homomorphism of Galois modules
$$\SS_i:=\SS/\ell^i\SS \to T_{\ell}(X)/\ell^i T_{\ell}(X)=X_{\ell^i}$$
is an injection of free $\Z/\ell^i\Z$-modules. Further, we will
identify $\SS_i$ with its image in $X_{\ell^i}$. We write
$\bar{e}^{i}_j$ for the image of $e_j$ in $\S_i \subset X_{\ell^i}$;
clearly, the $d$-element set $\{\bar{e}^{i}_j\}_{j=1}^d$ is a basis
of the free $\Z/\ell^i\Z$-module $\S_i$. By Corollary \ref{almostss}
applied to $n=\ell^i$ and $W=\SS_i$ there exists $u_i \in \End(X)$
such that
$$r \SS_i \subset u_i(X_{\ell^i}) \subset \SS_i.$$
In particular, $u_i(X_{\ell^i})$ contains $r \bar{e}^{i}_j$ for all
$j$. It is also clear that for each $z \in T_{\ell}(X)$
$$u_i(z)\in \SS+\ell^i T_{\ell}(X).$$

 For each $\bar{e}^{i}_j$ pick an element
 $$\bar{z}^{i}_j\in T_{\ell}(X)/\ell^i T_{\ell}(X)=X_{\ell^i}$$
 such that $u_i(\bar{z}^{i}_j)=r \bar{e}^{i}_j$. Let us pick
${z}^{i}_j\in T_{\ell}(X)$ such that its image in $X_{\ell^i}$
coincides with $\bar{z}^{i}_j$. Clearly, the image of
$u_i({z}^{i}_j)$ in $T_{\ell}(X)/\ell^i T_{\ell}(X)=X_{\ell^i}$
 equals $r \bar{e}^{i}_j$. Using the compactness of
$\End(X)\otimes \Z_{\ell}$ and $T_{\ell}(X)$, let us choose an
infinite increasing sequence of positive integers $i_1<i_2 < \dots <
i_m < \dots$ such that $\{u_{i_m}\}_{m=1}^{\infty}$ converges in
$\End(X)\otimes \Z_{\ell}$ to some $u$ and
$\{{z}^{j}_{i_m}\}_{m=1}^{\infty}$ converges in
$T_{\ell}(X)$ to some $z^{j}$ for all $j$ with $1 \le j\le d$. It
follows that
$$u(z^{j})=\lim u_{i_m}({z}^{j}_{i_m})=r e_j.$$
This implies that $u(T_{\ell}(X)) \supset r\cdot \S$. On the other
hand, for each $z \in T_{\ell}(X)$
$$u_{i_m}(z) \in \S +\ell^{i_m} T_{\ell}(X).$$
Since $\{i_m\}_{m=1}^{\infty}$ is an increasing set of positive
integers, $\{\SS +\ell^{i_m} T_{\ell}(X)\} _{m=1}^{\infty}$ is a
decreasing set of compact sets whose intersection is compact $\SS$.
It follows that $u(z)=\lim u_{i_m}(z)$ lies in $\SS$.
\end{proof}

\begin{lem}
\label{jacobsonL} Let $X$ be an abelian variety of positive
dimension over $K$. Let $\ell$ be a prime that is different from
$\fchar(K)$ and such that the $\Gal(K)$-module $X_{\ell}$ is
semisimple and
$$\End_{\Gal(K)}(X_{\ell})=\End(X)\otimes \Z/\ell\Z.$$
Then the centralizer of $\End(X)\otimes\Z_{\ell}$ in
$\End_{\Z_{\ell}}(T_{\ell}(X))$ coincides with
$\Z_{\ell}G_{\ell,X,K}$.
\end{lem}

\begin{proof}[Proof of Lemma \ref{jacobsonL}]
Clearly, $\tilde{G}_{X,\ell,K}$ is the image of
$\Gal(K)\twoheadrightarrow G_{\ell,X,K} \to
\Aut_{\F_{\ell}}(X_{\ell})$. It follows that the
$\tilde{G}_{X,\ell,K}$-module $X_{\ell}$ is semisimple and
$$\End_{\tilde{G}_{X,\ell,K}}(X_{\ell})=\End(X)\otimes \Z/\ell\Z.$$
By the Jacobson density theorem, $\F_{\ell}\tilde{G}_{X,\ell,K}$
coincides with the centralizer of $\End(X)\otimes \Z/\ell\Z$ in
$\End_{\F_{\ell}}(X_{\ell})$. (Here $\F_{\ell}\tilde{G}_{X,\ell,K}$
is the $\F_{\ell}$-subalgebra of $\End_{\F_{\ell}}(X_{\ell})$
spanned by $\tilde{G}_{X,\ell,K}$.)

Let $M$ be the centralizer of $\End(X)\otimes\Z_{\ell}$ in
$\End_{\Z_{\ell}}(T_{\ell}(X))$. Clearly, $M$ is a saturated
$\Z_{\ell}$-submodule of $\End_{\Z_{\ell}}(T_{\ell}(X))$ (i.e., the
quotient $\End_{\Z_{\ell}}(T_{\ell}(X))/M$ is torsion-free); in
addition, $M$ contains $\Z_{\ell}G_{\ell,X,K}$. We have
$$M/\ell M \subset
\End_{\Z_{\ell}}(T_{\ell}(X))\otimes \Z_{\ell}/\ell
\Z_{\ell}=\End_{\F_{\ell}}(X_{\ell}).$$ Clearly, $M/\ell M$ lies in
the centralizer of
$$\End(X)\otimes\Z_{\ell}=\otimes \Z_{\ell}/\ell
\Z_{\ell}=\End(X)\otimes \Z/\ell\Z.$$ This implies that
$$M/\ell M \subset \F_{\ell}\tilde{G}_{X,\ell,K}\subset
\End_{\F_{\ell}}(X_{\ell}).$$ On the other hand, the image of
$\Z_{\ell}G_{\ell,X,K}$ in
$$\End_{\Z_{\ell}}(T_{\ell}(X))\otimes \Z_{\ell}/\ell
\Z_{\ell}=\End_{\F_{\ell}}(X_{\ell})$$
 obviously coincides with $\F_{\ell}\tilde{G}_{X,\ell,K}$. Since
 this image lies in $M/\ell M$, we conclude that $M/\ell=\F_{\ell}\tilde{G}_{X,\ell,K}$
 and $M=\Z_{\ell}G_{\ell,X,K}+\ell\cdot M$. It follows from
 Nakayama's Lemma that the $\Z_{\ell}$-module $M$ coincides with its submodule $\Z_{\ell}G_{\ell,X,K}$.
\end{proof}

\begin{rem}
\label{finalTate}
Let $P$ be a singleton $\{\ell\}$ and $d=1$. Now  Theorem \ref{finiteTate}  combined with Corollary \ref{TateHomo} and Theorem \ref{ssTateP} implies readily
Theorem \ref{TateP}.
\end{rem}

\section{Proof of main results}
\label{mainProof} Throughout this section, $K$ is a field that is
finitely generated over $\F_p$. Let us put $P=\P\setminus \{p\}$.

\begin{proof}[Proof of Corollary \ref{semisimpleF} and   Theorem
\ref{endoN}]
Corollary \ref{semisimpleF} follows readily from Corollary
\ref{almostss} combined with Theorem \ref{finiteTate}. Theorem
\ref{endoN} follows readily from Theorem \ref{endoNP} combined with
Theorem \ref{finiteTate}.
\end{proof}

\begin{proof}[Proof of Theorems \ref{finiteisogp} and
\ref{homoN}]  One has only to combine Theorem \ref{finiteTate} with
Theorems \ref{finiteisog} and \ref{homoNP} respectively.
\end{proof}

\begin{proof}[Proof of Theorem \ref{jacobsonT}]
The assertion follows readily from Lemma \ref{jacobsonL} combined
with Corollary \ref{semisimpleF} and Theorem \ref{endoN}.
\end{proof}

\begin{proof}[Proof of Theorem \ref{CM}]
The proof of \cite[Theorem 4.7.4]{ZarhinMZ3} works literally
provided one replaces the reference to \cite[Theorem
1.1.1]{ZarhinMZ2} by references to Corollaries \ref{semisimpleF} and
\ref{tateFiniteC}.
\end{proof}

\begin{proof}[Proof of Theorem \ref{maxab}]
Theorem \ref{maxab} is an immediate corollary of the conjunction of
two following statements. (Compare with Theorems 2 and 3 on p. 133
of \cite{ZarhinDUKE}.)

\begin{thm}
\label{maxabprimeL} Let $X$ be a simple abelian variety over $K$
that is not of CM type. Let $\ell\ne p$ be a prime, $W$ a nonzero
Galois-invariant $\Q_{\ell}$-vector space in $V_{\ell}(X)$ and $G_W$
the image of $\Gal(K)$ in $\Aut_{\Q_{\ell}}(W)$. Then the group
$G_W$ is not commutative.
\end{thm}

\begin{thm}
\label{maxabprimeELL} Let $X$ be a simple abelian variety over $K$
that is not of CM type. Then for all but finitely many primes
$\ell\ne p$ the following condition holds: Let $W$ a nonzero
Galois-invariant $\F_{\ell}$-vector space in $X_{\ell}$ and $G_W$
the image of $\Gal(K)$ in $\Aut_{\F_{\ell}}(W)$. Then the group
$G_W$ is not commutative.
\end{thm}

\begin{proof}[Proof of Theorems \ref{maxabprimeL} and \ref{maxabprimeELL}]
 The proof of  Theorems 2 and 3 in \cite[Sect.  3]{ZarhinDUKE} works literally
provided one replaces the reference to  \cite[p. 139, Statements 1
and 2]{ZarhinDUKE}  by references   to Theorem \ref{TateP} (instead
of Statement 1) and
 to Corollaries \ref{semisimpleF} and
\ref{tateFiniteC} (instead of Statement 2).
\end{proof}
\end{proof}

\begin{proof}[Proof of Theorems \ref{ConstEll} and \ref{ConstRoot}]
The proofs of Theorems 7 and 8 in \cite[Sect.  4]{ZarhinDUKE} work
literally in our case for Theorems \ref{ConstEll} and
\ref{ConstRoot} respectively. (As in the proof  of Theorems
\ref{maxabprimeL} and \ref{maxabprimeELL} one should replace the
reference to  \cite[p. 139, Statements 1 and 2]{ZarhinDUKE}  by
references   to Theorem \ref{TateP}  and to Corollaries
\ref{semisimpleF} and \ref{tateFiniteC}.)
\end{proof}

\section{Torsion and ramification in solvable extensions}
\label{lifting}
\begin{sect}
\label{ramification} Let $K$ be an arbitrary field and $\Oc \subset
K$ a discrete valuation ring whose field of fractions coincides
with $K$. We write $\p$ for the maximal ideal of $\Oc$ and $p$ for
the characteristic of the residue field $\Oc/\p$. Let $L/K$ be a
finite Galois field extension with Galois group $\Gal(L/K)$ and
$\Oc_L$ the integral closure of $\Oc$ in $L$. The following
assertion is well known (see, e.g., \cite[Ch. 5, Sect. 7--10]{ZS},
\cite[Ch. 1, Sect. 7]{SerreCL}).

\begin{itemize}
\item[(i)]
$\Oc_L$ is a principal ideal domain, the set of its maximal ideals
is finite and consists of (say) $g$ maximal ideals $\q_1, \dots ,
\q_g$ such that
$$\p\Oc=\left(\prod_{i=1}^g \q_i\right)^{e(L/K)}$$ where $e(L/K)$ is
the (weak) ramification index at  $\p$. The degree
$f(L/K)=[\Oc_L/\q_i: \Oc/\p]$ of the field extension
$(\Oc_L/\q_i)/(\Oc/\p)$ equals the product $f_0 p^s$ where $f_0$ is
the degree of the separable closure  of $\Oc/\p$ in $\Oc_L/\q_i$ and
$s$ is a nonnegative integer that vanishes if and only if
$\Oc_L/\q_i$ is separable over $\Oc/p$. The integers $f_0$ and $s$
do not depend on $i$.
 The product
$$e(L/K)\cdot f_0 p^s g= e(L/K) \cdot f \cdot g=[L:K]=\#(\Gal(L/K))$$
In particular, $e(L/K)\cdot p^s$ divides $[L:K]$. The field
extension $L/K$ is tamely ramified at $p$ if and only if
$\#(I(\q_i))$ is not divisible by $p$, i.e., $e(L/K)$ is not
divisible by $p$ and $\Oc_L/\q_i$ is separable over $\Oc/p$. Here
$$I(\q_i)\subset \Gal(L/K)$$
 is the inertia subgroup attached to $\q_i$.
\item[(ii)]
The Galois group $\Gal(L/K)$ acts transitively on the set
$\{\q_i\mid 1 \le i \le g\}$.  The corresponding inertia subgroups
$I(\q_i)\subset \Gal(L/K)$ are conjugate subgroups of order $e(L/K)
p^s$  in $\Gal(L/K)$. (See \cite[Ch. 5, Sect. 10, Th. 24 and its
proof]{ZS}.)
\item[(iii)]
Let $L_0/K$ be a Galois subextension of $L/K$, i.e., $L_0/K$ is a
Galois field extension and $L_0\subset L$. Then $\q_i^{0}=\q_i
\bigcap \Oc_{L_0}$ is a maximal ideal in $\Oc_{L_0}$ that lies
above $\p$. The image of $I(\q_i)$ under the surjection $\Gal(L/K)
\twoheadrightarrow \Gal(L_0/K)$ coincides with the inertia subgroup
$$I(\q_i^{0})\subset \Gal(L_0/K)$$
attached to $\q_i^{0}$
\cite[Ch. 1, Sect. 7, Prop. 22(b)]{SerreCL}; in
particular, $\#(I(\q_i^{0}))$ divides $\#(I(q_i))$. On the other
hand, $\#(I(\q_i^{0}))$ divides $\#(\Gal(L_0/K))=[L_0:K]$.
 This implies that if $[L_0:K]$ and $\#(I(q_i))$ are relatively
prime then $\#(I(\q_i^{0}))=1$, i.e., $L_0/K$ is {\sl unramified} at
$\p$.
\end{itemize}

Let $Y$ be an abelian variety of positive dimension over $K$. Let
$n\ge 3$ be an integer that is {not} divisible by $p$. Assume that
$Y_n \subset Y(K)$, i.e., all points of order $n$ on $X$ are defined
over $K$.

Let $\Y \to \Spec(\Oc)$ be the N\'eron model of $Y$ \cite{Neron}; it
is a smooth group scheme whose generic fiber coincides with $Y$.
Since $Y_n \subset Y(K)$, the Raynaud criterion \cite[Prop.
4.7]{GrothendieckN} tells us that $Y$ has {\sl semistable} reduction
at $\p$, i.e., $\Y$ is a semiabelian group scheme.

Let $\ell$ be a prime different from $p$. For all positive integers
$j$ the field $K(Y_{m_j})$ with $m_j=\ell^j$ is a finite Galois
extension. The following assertion was inspired by \cite[Ch. XII,
Sect. 2.0 on p. 242]{MB}.

\end{sect}

\begin{lem}
\label{ramIndex} Let $n\ge 3$ be an integer that is {not} divisible
by $p$. Assume that $Y_n \subset Y(K)$.
 Let $L=K(Y_{\ell})$ and
$\q$ be a maximal ideal in $\Oc_{L}$, which lies above $\p$. Then:

\begin{enumerate}
\item[(i)]  The inertia group $I(\q)$ is a finite
commutative $\ell$-group. In particular,  $L/K$ is tamely ramified
at $\p$.
\item[(ii)]
Let $L_0/K$ be a Galois subextension of $L/K$. If $[L_0:K]$ is not
divisible by $\ell$ then $L_0/K$ is unramified at $\p$.
\end{enumerate}
\end{lem}

\begin{proof}[Proof of Lemma \ref{ramIndex}]
The assertion (ii) follows readily from (i). So, let us prove (i).

 For all positive integers $j$ let us put $L_j=K(Y_{m_j})$ and
$\Oc_j=\Oc_{L_j}$. We have $L_1=L, \Oc_1=\Oc_{L}$.
 We have a tower of Galois extensions
$$K \subset L_1 \subset L_2 \subset \dots L_j \subset \dots .$$
We write $L_{\infty}$ for the union $\bigcup_{i=1}^{\infty} L_i$.
For each $j$ pick a maximal ideal $\q^{(j)}$ in $\Oc_j$ in such a
way that $\q^{(1)}=\q$ and $\q^{(j+1)}$ lies above $\q^{(j)}$. (Such
a choice is possible, because the projective limit of nonempty
finite sets is nonempty.)  Then the Galois group
$\Gal(L_{\infty}/K)$ is the projective limit of finite groups
$\Gal(L_j/K)$'s. It is also clear that
$$\Gal(L_j/K)=\bar{\rho}_{m_j,Y}(\Gal(K)) \subset
\Aut_{\Z/m_j\Z}(Y_{m_j}),$$
$$\Gal(L_{\infty}/K)=\rho_{\ell,Y}(\Gal(K))\subset
\Aut_{\Z_{\ell}}(T_{\ell}(Y))\subset
\Aut_{\Q_{\ell}}(V_{\ell}(Y)).$$

Recall that the natural homomorphisms $I(\q^{(j+1)})\to I(\q^{(j)})$
are surjective for all $j$.
 Let $I_{\infty}$ be the projective limit
of the corresponding inertia subgroups $I(\q^{(j)})$; clearly,
$I_{\infty}$ is a compact subgroup of $\Gal(L_{\infty}/K)$ and for
each $j$ the natural group homomorphism $I_{\infty} \to I(\q^{(j)})$
is surjective, because the projective limit of nonempty finite sets
is also nonempty. Therefore one may view $I_{\infty}$ as a certain compact
subgroup of
$$\Aut_{\Z_{\ell}}(T_{\ell}(Y))\subset \Aut_{\Q_{\ell}}(V_{\ell}(Y)).$$
Since $Y$ has semistable reduction at $\q$, there exists a
$\Q_{\ell}$-vector subspace $\W \subset V_{\ell}(Y)$ such that
$I_{\infty}$ acts trivially on $\W$ and $V_{\ell}(Y)/\W$ \cite[Prop.
3.5]{GrothendieckN}. It gives us an injective continuous
homomorphism of topological groups $$I_{\infty} \to
\Hom_{\Q_{\ell}}(V_{\ell}(Y)/\W), \W),$$
$$\sigma \mapsto \{v+\W \mapsto \sigma(v)-v\} \ \forall \ \sigma \in I_{\infty}, \ v \in V_{\ell}(Y).$$
Since $I_{\infty}$ is compact,
there is a continuous isomorphism between
$I_{\infty}$ and its image, which is a compact subgroup of
$\Hom_{\Q_{\ell}}(V_{\ell}(Y)/\W, \W)$. Since the latter is a
finite-dimensional $\Q_{\ell}$-vector space, all of its  compact
subgroups  are commutative pro-$\ell$-groups (that are isomorphic
either to a direct sum of several copies of $\Z_{\ell}$ or to zero).
It follows that $I_{\infty}$ is also a commutative pro-$\ell$-group.
 Since there is a surjective continuous homomorphism $$I_{\infty}
 \to I(\q^{(j)}),$$ every $I(\q^{(j)})$ is a finite commutative $\ell$-group.
 This ends the proof, since $\q=\q^{(1)}$ and therefore $I(\q)=I(\q^{(1)})$.
\end{proof}

\begin{proof}[Proof of Theorem \ref{fongswan}]
If $K$ is finite then there is nothing to prove. So, let us assume
that $K$ is infinite. Let $d\ge 1$ be  the transcendence degree of
$K$ over $\F_p$. Let us pick a positive integer $n \ge 3$ that is
not divisible by $p$. Replacing $K$ by $K(X_n)$ and $X$ by $X
\times_K K(X_n)$ , we may assume that $X_n\subset X(K)$.

Let $P$ be an {\sl infinite} set of primes $\ell \ne p$ such that
$\tilde{G}_{\ell,X,K}$ is $\ell$-solvable. By deleting finitely many
primes from $P$, we may and will assume that the $\Gal(K)$-module
$X_{\ell}$ is semisimple for all $\ell \in P$. Since
$\tilde{G}_{\ell,X,K}$ is the image of $\Gal(K)$ in
$\Aut_{\F_{\ell}}(X_{\ell})\cong \GL(2g,\F_{\ell})$, the
$\tilde{G}_{\ell,X,K}$-module $X_{\ell}$ is semisimple for all
$\ell\in P$.

Let $\C$ be the field of complex numbers. Let us put $g=\dim(X)$.
Recall that $X_{\ell}$ is a
 $2g$-dimensional $\F_{\ell}$-vector space. Let $G$ be a finite $\ell$-solvable subgroup of
$\Aut_{\F_{\ell}}(X_{\ell})$ such that the natural faithful
representation of $G$ in $X_{\ell}$ is completely reducible. Let us
split the semisimple $\F_{\ell}[G]$-module $X_{\ell}$ into a direct
sum
$$X_{\ell}=\oplus_{i=1}^m W_i$$
of simple $\F_{\ell}[G]$-modules $W_i$. If
$d_i=\dim_{\F_{\ell}}(W_i)$ then $2g==\sum_{i=1}^m d_i$. By a
theorem of Fong-Swan \cite[Sect. 16.3, Th. 38]{SerreRep},  each
$W_i$ lifts to characteristic zero; in particular,  there is a group
homomorphism $\rho_i:G \to \GL(d_i,\C)$ whose kernel lies in the
kernel of $G \to \Aut_{\F_{\ell}}(W_i)$ (for all $i$). Clearly, the
product-homomorphism
$$\rho=\prod_{i=1}^d \rho_i: G \to \prod_{i=1}^m \GL(d_i,\C)\subset
\GL(2g,\C)$$ is an injective group homomorphism and therefore $G$ is
isomorphic to  a finite subgroup of $\GL(2g,\C)$. By a theorem of
Jordan \cite[Th. Th. 36.13]{CR}, there is a positive integer
$N=N(2g)$ that depends only on $2g$ and such that $G$ contains a
normal abelian subgroup of index dividing $N$. By deleting from $P$
all prime divisors of $N$, we may and will assume that $\ell$ does
{\sl not} divide $N$ for each $\ell\in P$.

Let us apply this observation to $G=\tilde{G}_{\ell,X,K}$. We obtain
that for all $\ell \in P$   the group $\tilde{G}_{\ell,X,K}$
contains an abelian normal subgroup $H_{\ell}$ of index dividing
$N$. For each  $\ell\in P$ let us consider the corresponding
subfield of $H_{\ell}$-invariants
$$K_{\ell}:=\left(K(X_{\ell})\right)^{H_{\ell}}\subset K(X_{\ell}).$$
Clearly, $K_{\ell}/K$ is a Galois extension of degree dividing $N$
while $\Gal(K(X_{\ell})/K_{\ell})$ coincides with commutative
(sub)group $H_{\ell}$. Since $K_{\ell}\subset K(X_{\ell})$ and
$[K_{\ell}:K]$ divides $N$  and therefore is {\sl not} divisible by $\ell$, Lemma
\ref{ramIndex} tells us that the Galois extension $K_{\ell}/K$ is
unramified with respect to every discrete valuation of $K$. (Since
$\fchar(K)=p$, its every residual characteristic is also $p$.)

Let $k$ be the (finite) algebraic closure of $\F_p$ in $K$ and let
$S$ be an absolutely irreducible normal $d$-dimensional projective
variety over $k$ whose field of rational functions $k(S)$ coincides
with $K$. Let $V\subset S$ be a {\sl smooth} open dense subset such
that the codimension of $S\setminus V$ in $S$ is, at least, $2$. Let
$V_{\ell}$ be the normalization of $V$ in $K_{\ell}/K$. Now
Zariski-Nagata purity theorem tells us that the regular map
$V_{\ell} \to V$ is an \'etale Galois cover; clearly, its degree
equals $[K_{\ell}:K]$ and therefore divides $N$.  Let $\pi_1(V)$ be
 the  fundamental group of $V$
 that classifies  \'etale covers of
$V$ (\cite{GrothendieckM}, \cite[Ch. XII, Sect. 1, pp.
241--242]{MB}. This group is a (natural) topologically finitely
generated (topological) quotient of $\Gal(K)$ \cite[Ch. XII, Sect.
1, p. 242]{MB} and the natural surjection
$$\Gal(K) \twoheadrightarrow \Gal(K_{\ell}/K)$$
factors through $\pi_1(V)$, i.e., it is the composition of the
canonical continuous surjection $\Gal(K) \twoheadrightarrow \pi_1(V)$ and a certain
continuous surjective homomorphism
$$\gamma_{\ell}:\pi_1(V) \twoheadrightarrow
\Gal(K_{\ell}/K)$$ whose kernel is an open normal subgroup in $\pi_1(V)$ of index
dividing $N$.

Since  $\pi_1(V)$ is topologically finitely generated, it contains
only finitely many open normal subgroups of index dividing $N$
(because it admits only finitely many continuous homomorphisms to
any finite group of order dividing $N$). If $\Gamma$ is the
intersection of all such subgroups then it is an open normal
subgroup in $\pi_1(V)$ that lies in the kernel of every
$\gamma_{\ell}$ for all $\ell \in P$; in particular, it has finite
index. The preimage $\Delta$ of $\Gamma$ is an open normal subgroup
of finite index in $\Gal(K)$ and the corresponding subfield of
$\Delta$-invariants $E:=\bar{K}_s^{\Delta}$ is a finite Galois
extension of $K$ that contains  $K_{\ell}$ for all $\ell \in P$.
This implies that for all $\ell \in P$ the compositum $E
K(X_{\ell})$ is abelian over $E$, because $K(X_{\ell})$ is abelian
over $K_{\ell}$. But $E K(X_{\ell})=E({X^E}_{\ell})$ where
$$X^E=X\times_{K} E$$ is an abelian
variety over $E$. Applying Theorem \ref{CM} to $X^E$ and $E$
(instead of $X$ and $K$), we conclude that $X^E$ is an abelian
variety of CM type and isogenous over $\bar{E}=\bar{K}$ to an
abelian variety that is defined over a finite field. The same is
true for $X$, since $X^E=X \times_K E$.

\end{proof}

\section{Concluding Remarks}
\label{conclude}

Theorem \ref{TateP} and Corollaries \ref{semisimpleF}(iii) and
Corollary \ref{tateFiniteC} imply readily that the following results
remain true for  all prime characteristics $p$,  including $p=2$.

\begin{itemize}

\item
 Assertions  (Sect. 1.3 and
4.4), Corollaries 1--6, Theorem 4.1, and Remark 1 of
\cite{ZarhinMZ3} remain true over any field $E$ that is finitely
generated over a finite field of arbitrary characteristic, including
$2$. Corollary 7 of \cite{ZarhinMZ3} remains true for any field $F$
of arbitrary prime characteristic, including $2$.
\item
Theorem 1.1(ii) of \cite{SZ}.

\item
Theorems 1.1, 1.4, 1.6, 1.9, 2.1  of \cite{ZarhinJussieu}.

\end{itemize}

\section{Corrigendum to \cite{ZarhinG}}

Page 317, Remark 1.3, second sentence: one has to assume
additionally that the kernel of the morphism is $W$.

Page 317, line -12: read $\Hom(Y^t,X^t)$ instead of $\Hom(Y,X)$.

Page 326, Proof of Theorem 3.4, second line: read $8g$-dimensional
instead of $4g$-dimensional.


\begin{thebibliography}{99}



 \bibitem{Neron}  Bosch S.,  L\"utkebohmert W.,  Raynaud M., N\'eron
 models, Springer Verlag, Berlin Heidelberg New York, 1990.


 \bibitem{CR}  Curtis Ch. W.,  Reiner I., Representation theory of
finite groups and associative algebras, Interscience Publishers, New
York London, 1962.



\bibitem{Faltings2} G. Faltings,  Complements to Mordell,
Chapter VI in:  Faltings G.,  Wustholz G. et al., Rational points.
Aspects of Mathematics,  E6. Friedr. Vieweg \& Sohn, Braunschweig,
1984.




\bibitem{GrothendieckM}  Grothendieck A et al., Rev\^etements \'etales et groupe fondamental (SGA 1),
Lecture Notes in Math. {\bf 224}, Springer Verlag,
Berlin-Heidelberg-New York, 1971.

\bibitem{GrothendieckN}  Grothendieck A.,   Mod\`eles de N\'eron et
monodromie, Expose IX dans  SGA7 I, Lecture Notes in Math. {\bf
288}, Springer Verlag, Berlin Heidelberg New York, 1972,
313--523.

\bibitem{Lang}  Lang S., Algebra, Addison-Wesley Publishing Company, Reading, MA, 1965.



\bibitem{LOZ}  Lenstra H. W.,Jr,  Oort F., Zarhin  Yu. G.,  Abelian subvarieties,
 J. Algebra,  1996, 180, 513--516.


\bibitem{MB}  Moret-Bailly L., Pinceaux de vari\'et\'es ab\'eliennes, Ast\'erisque,
vol. 129 (1985).


\bibitem{Mumford}  Mumford D., Abelian varieties, Second edition,
 Oxford University Press, London, 1974.

 \bibitem{OortSM}  Oort F.,  The isogeny class of a CM-type abelian variety is
 defined over a finite extension of the prime field, J. Pure
 Applied Algebra,  1973,  3, 399--408.




\bibitem{SerreCL}  Serre J.-P.,  Corps Locaux, Troisieme edition corrig\'ee, Hermann, Paris 1968.

\bibitem{SerreIzv}  Serre J.-P.,   Sur les groupes des congruence des
vari\'et\'es ab\'eliennes. Izv. Akad. Nauk SSSR ser. matem., 1964,
28, 3-18; \OE uvres  {\bf  II}, pp. 230--245,
Springer Verlag, Berlin, 1986.

\bibitem{Serre}  Serre J-P., Abelian $\ell$-adic representations
 and elliptic curves. Second edition. Addison-Wesley, New York, 1989.

\bibitem{SerreRep}  Serre J-P., Repr\'esentations lin\'eares des
groupes finis, Troisieme edition corrig\'ee, Hermann, Paris, 1978.




\bibitem{SZ}  Skorobogatov  A.N.,  Zarhin Yu. G.,  A finiteness theorem for Brauer groups of abelian varieties and K3 surfaces.  J. Algebraic Geometry,  2008, 17, 481--502.

\bibitem{Tate}  Tate J., Endomorphisms of Abelian varieties over finite
fields, Invent. Math.   2, 1966, 134--144.


\bibitem{ZarhinIz}  Zarhin Yu. G.,  Endomorphisms of Abelian varieties over fields of finite
characteristic, Izv. Akad. Nauk SSSR ser. matem.,  1975, 39,
272--277; Math. USSR Izv. 1975, 9, 255--260.

\bibitem{ZarhinMZ1} Zarhin  Yu. G.,  Abelian varieties in characteristic
$P$, Mat. Zametki, 1976, 19, 393--400; Mathematical Notes,
1976, 19, 240--244.

\bibitem{ZarhinMZ2} Zarhin Yu. G.,   Endomorphisms of Abelian varieties and points of finite order
in characteristic $P$, Mat. Zametki, 1977,  21, 737--744;
Mathematical Notes,
1978,  21, 415--419.

 \bibitem{ZarhinMZ3} Zarhin  Yu. G.,   Torsion of abelian varieties in
fininite characteristic, Mat. Zametki,  1977,  22, 1--11;
       Mathematical Notes, 1978,  22, 493--498.


\bibitem{ZarhinInv1979}  Zarhin, Yu. G., Abelian varieties, $\ell$-adic representations and Lie algebras.
Rank independence on $\ell$, Invent. Math., 1979,  55, 165--176.

\bibitem{ZarhinYar}  Zarhin,  Yu. G.,  Homomorphisms of Abelian varieties and points of finite order over
       fields of finite characteristic  (in  Russian),
       In: A. L. Onishchik (Ed.), Problems in Group Theory and Homological Algebra
        Yaroslavl Gos. Univ., Yaroslavl, 1981, 146--147; MR0709632 (84m:14051).



\bibitem{ZarhinInv} Zarhin Yu. G.,  A finiteness theorem for unpolarized Abelian varieties over
number fields with prescribed places of bad reduction, Invent. Math.
 1985,  79, 309--321.

 \bibitem{ZarhinDUKE} Zarhin Yu. G., Endomorphisms and torsion of abelian varieties. Duke Math. J., 1987, 54, 131--145.

\bibitem{ZarhinP} Zarhin Yu. G.,  Parshin A. N.,  Finiteness problems in Diophantine
geometry.  Amer. Math. Soc. Transl. (2), 1989,  143, 35--102;
arXiv 0912.4325 [math.NT] .

\bibitem{ZarhinMRL2}  Zarhin Yu. G., Hyperelliptic Jacobians without complex multiplication in positive
    characteristic, Math. Research Letters,  2001, 8, 429--435.

\bibitem{ZarhinSb}    Zarhin Yu. G.,  Endomorphism rings of certain jacobians in finite
characteristic. Matem. Sbornik,    2002, 193:8, 39--48;
 Sbornik Math., 2002,  193:8, 1139-1149.



\bibitem{ZarhinBSMF} Zarhin  Yu. G.,  Non-supersingular hyperelliptic jacobians, Bull.  Soc. Math.  France,
 2004, 132, 617--634.



\bibitem{ZarhinG} Yu. G.~Zarhin, {\sl Homomorphisms of abelian varieties over finite fields}, pp.  315--343. In: D. Kaledin, Yu. Tschinkel (Eds.),
Higher-dimensional geometry over finite fields,  IOS Press, Amsterdam, 2008.


\bibitem{ZarhinJussieu}  Zarhin Yu. G.,  Homomorphisms of abelian varieties over geometric fields of
finite characteristic, J. Inst. Math. Jussieu, 2013, 12, 225--236.

\bibitem{ZS} Zariski  O., Samuel  P., Commutative Algebra, vol. I, Van
Nostrand, Princeton, NJ 1956.
\end{thebibliography}
\end{document}